\documentclass[11pt]{article}
\usepackage{amsmath}
\usepackage{amsfonts}
\usepackage{amssymb}
\usepackage{color}
\usepackage{latexsym}
\usepackage{multirow}
\usepackage{graphicx}
\newtheorem{theorem}{\bf Theorem}[section]
\newtheorem{lemma}[theorem]{\bf Lemma}
\newtheorem{prop}[theorem]{\bf Proposition}

\newtheorem{coro}[theorem]{\bf Corollary}

\newtheorem{defn}[theorem]{\bf Definition}
\newenvironment{proof}{\noindent{\em Proof:}}{\quad \hfill$\Box$\vspace{2ex}}

\def \bN {\Bbb N}

\def \bR {\Bbb R}

\def \bC {\Bbb C}

\def \bx {{\bf x}}

\def \bz {{\bf z}}

\def \cB {{\cal B}}

\def \cI {{\cal I}}
\def \cJ {{\cal J}}

\def \cL {{\cal L}}
\def \cM {{\cal M}}

\def \cE {{\cal E}}
\def \cN {{\cal N}}

\def \cR {{\cal R}}

\def \cW {{\cal W}}

\def \and {\, \mbox{\rm and}\, }

\def \span {\,{\rm span}\,}

\makeatletter

\newcommand{\Rmnum}[1]{\expandafter\@slowromancap\romannumeral #1@}
\makeatother
\begin{document}
\title{\bf Vector-valued Reproducing Kernel Banach Spaces \\ with Applications to Multi-task Learning\thanks{This work was partially supported by the US National
Science Foundation under grant 0631541 and by Guangdong Provincial Government of China through the ``Computational Science Innovative Research Team" program.}}
\author{Haizhang Zhang\thanks{School of Mathematics and Computational Science and Guangdong Province Key Laboratory of Computational Science, Sun Yat-sen University, Guangzhou 510275, P. R. China. E-mail address: {\it zhhaizh2@sysu.edu.cn.} The research was accomplished while the author was visiting University of Michigan.}\quad and \quad Jun
Zhang\thanks{Department of Psychology, University of Michigan, Ann Arbor, MI 48109, USA. E-mail address: {\it junz@umich.edu}.}}
\date{}
\maketitle
\begin{abstract} Motivated by multi-task machine learning with Banach spaces, we propose the notion of vector-valued reproducing kernel Banach spaces (RKBS). Basic properties of the spaces and the associated reproducing kernels are investigated. We also present feature map constructions and several concrete examples of vector-valued RKBS. The theory is then applied to multi-task machine learning. Especially, the representer theorem and characterization equations for the minimizer of regularized learning schemes in vector-valued RKBS are established.

\end{abstract}

\section{Introduction}
\setcounter{equation}{0}

The purpose of this paper is to establish the notion of vector-valued reproducing kernel Banach spaces and demonstrate its applications to multi-task machine learning. Built on the theory of scalar-valued reproducing kernel Hilbert spaces (RKHS) \cite{Aronszajn}, kernel methods have been proven successful in single task machine learning \cite{CuckerSmale,EPP,ScSm,ShCr,Vapnik}. Multi-task learning where the unknown target function to be learned from finite sample data is vector-valued appears more often in practice. References \cite{EMP,MP2005} proposed the development of kernel methods for learning multiple related tasks simultaneously. The mathematical foundation used there was the theory of vector-valued RKHS \cite{BM,Pedrick}. Recent progresses in vector-valued RKHS can be found in \cite{CMPY,CDT,CDTU}. In such a framework, both the space of the candidate functions used for approximation and the output space are chosen as a Hilbert space.

There are some occasions where it might be desirable to select the space of candidate functions, the output space, or both as Banach spaces. Hilbert spaces constitute a special and limited class of Banach spaces. Any two Hilbert spaces over a common number field with the same dimension are isometrically isomorphic. By reaching out to other Banach spaces, one obtains more variety in geometric structures and norms that are potentially useful for learning and approximation. Moreover, training data might come with intrinsic structures that make them impossible or inappropriate to be embedded into a Hilbert space. Learning schemes based on features in a Hilbert space may not work well for them. Finally, in some applications, a Banach space norm is engaged for some particular purpose. A typical example is the linear programming regularization in coefficient based regularization for machine learning \cite{ScSm}, where the $\ell_1$ norm is employed to obtain sparsity in the resulting minimizer.

There have been considerable work in learning a single task with Banach spaces (see, for example, \cite{Bennett,CMR,DL,Gentile,HBS,KL,MP2004,MP2007,LB, Zhang,ZXZD}). The difficulty in mapping patterns into a Banach space and making use of these features for learning mainly lies in the lack of an inner product in Banach spaces. In particular, without an appropriate correspondence of the Riesz representation of continuous linear functionals, point evaluations do not have a kernel representation in these studies. Semi-inner products, a mathematical tool discovered by Lumer \cite{Lumer} for the purpose of extending Hilbert space type arguments to Banach spaces, seem to be a natural substitute for inner products in Banach spaces. An illustrative example is that we were able to extend the classical theory of frames and Riesz bases to Banach spaces via semi-inner products \cite{ZZacha}. Semi-inner products were first used to machine learning by Der and Lee \cite{DL} for the study of large margin classification by hyperplanes in a Banach space. With this tool, we established the notion of scalar-valued reproducing kernel Banach spaces (RKBS) and investigated regularized learning schemes in RKBS \cite{ZXZ,ZZjogo}. There has been increasing interest in the application of this new theory \cite{Zhdanov,Jorgensen,SongZhang,Sriperumbudur}.

We attempt to build a mathematical foundation for multi-task learning with Banach spaces. Specifically, we shall propose a definition of vector-valued RKBS and investigate its fundamental properties in the next section. Feature map representations and several concrete examples of vector-valued RKBS will be presented in Sections 3 and 4, respectively. In Section 5, we investigate regularized learning schemes in vector-valued RKBS.

\section{Definition and Basic Properties}
\setcounter{equation}{0}

We are concerned with spaces of functions from a fixed set to a vector space. We shall allow the space of functions and the range space both to be a Banach space. Our key tool in dealing with a general Banach space is the semi-inner product \cite{Giles,Lumer}. Recall that a semi-inner product on a Banach space $V$ is a function from $V\times V$ to $\bC$, denoted by $[\cdot,\cdot]_V$, such that for all $u,v,w\in V$ and $\alpha,\beta\in\bC$
\begin{enumerate}
\item (linearity with respect to the first variable) $[\alpha f+\beta g,h]_V=\alpha[f,h]_V+\beta[g,h]_V$;
\item (positivity) $[f,f]_V>0$ for $f\ne0$;
\item (conjugate homogeneity with respect to the second variable) $[f,\alpha g]_V=\overline{\alpha}[f,g]_V$;
\item (Cauchy-Schwartz inequality) $|[f,g]_V|\le [f,f]_V^{1/2}[g,g]_V^{1/2}$.
\end{enumerate}
A semi-inner product $[\cdot,\cdot]_V$ on $V$ is said to be {\it compatible} if
$$
[f,f]_V^{1/2}=\|f\|_V\mbox{ for all }f\in V,
$$
where $\|\cdot\|_V$ denotes the norm on $V$. Every Banach space has a compatible semi-inner product \cite{Giles,Lumer}. Let $[\cdot,\cdot]_V$ be a compatible semi-inner product on $V$. Then one sees by the Cauchy-Schwartz inequality that for each $f\in\cB$, the linear functional $f^*$ on $V$ defined by
\begin{equation}\label{dualitymapping}
f^*(g):=[g,f]_V,\ \ g\in V
\end{equation}
is bounded on $V$. In other words, $f^*$ lies in the dual space $\cB^*$ of $\cB$. Moreover, we have
\begin{equation}\label{equalnormduality}
\|f^*\|_{V^*}=\|f\|_V
\end{equation}
and
\begin{equation}\label{dualityrelation}
f^*(f)=\|f\|_V\|f^*\|_{V^*}.
\end{equation}
Introduce the {\it duality mapping} $\cJ_V$ from $V$ to $V^*$ by setting
$$
\cJ_V(f):=f^*,\ \ f\in V.
$$

We desire to represent the continuous linear functionals on the vector-valued RKBS to be introduced by the semi-inner product. However, the semi-inner product might not be able to fulfill this important role for an arbitrary Banach space. For instance, one verifies that the continuous linear functional
$$
\mu(g):=\sum_{j=1}^\infty (-1)^j\frac1{2^j}g\left(\frac1{2^j}\right),\ \ g\in C([0,1]).
$$
on $C([0,1])$ endowed with the usual maximum norm can not be represented as
$$
\mu(g)=[g,f],\ \ g\in C([0,1])
$$
for any compatible semi-inner product $[\cdot,\cdot]$ on $C([0,1])$ and any $f\in C([0,1])$.

The above example indicates that the duality mapping might not be surjective for a general Banach space. Other problems such as non-uniqueness of compatible semi-inner products and non-injectivity of the duality mapping may also occur. To overcome these difficulties, we shall focus on Banach spaces that are uniformly convex and uniformly Fr\'{e}chet
differentiable in this preliminary work on vector-valued RKBS. A Banach space $V$ is {\it uniformly convex} if for all
$\varepsilon>0$ there exists a $\delta>0$ such that
$$
\|f+g\|_V\le 2-\delta \mbox{ for all }f,g\in V\mbox{ with
}\|f\|_{V}=\|g\|_{V}=1\mbox{ and }\|f-g\|_V\ge\varepsilon.
$$
Uniform convexity ensures the injectivity of the duality mapping and the existence and uniqueness of the best approximation to a closed convex subset of $V$ \cite{Giles}. We also say that $V$ is {\it uniformly Fr\'{e}chet
differentiable} if for all $f,g\in V$
\begin{equation}\label{diffentiability}
\lim_{t\in\bR,\,t\to 0}\frac{\|f+tg\|_V-\|f\|_V}t
\end{equation}
exists and the limit is approached uniformly for all $f,g$ in the unit ball of $V$. If $V$ is uniformly Fr\'{e}chet
differentiable then it has a unique compatible semi-inner product \cite{Giles}. The differentiability (\ref{diffentiability}) of the norm is useful to derive characterization equations for the minimizer of regularized learning schemes in Banach spaces. For simplicity, we call a Banach space {\it uniform} if it is both uniformly convex and uniformly Fr\'{e}chet
differentiable. An analogue of the Riesz representation theorem holds for uniform Banach spaces.

\begin{lemma}\label{riesz}(Giles \cite{Giles})
Let $V$ be a uniform Banach space. Then it has a unique compatible semi-inner product $[\cdot,\cdot]_V$ and the duality mapping $\cJ_V$ is bijective from $V$ to $V^*$. In other words, for each $\mu\in V^*$ there exists a unique $f\in V$ such that
$$
\mu(g)=[g,f]_V\mbox{ for all }g\in V.
$$
In this case,
\begin{equation}\label{sipondual}
[f^*,g^*]_{\cB^*}:=[g,f]_\cB,\ \ f,g\in\cB
\end{equation}
defines a compatible semi-inner product on $\cB^*$.
\end{lemma}

Let $V$ be a uniform Banach space. We shall always denote by $[\cdot,\cdot]_V$ the unique compatible semi-inner product on $V$. By Lemma \ref{riesz} and equation (\ref{equalnormduality}), the duality mapping is bijective and isometric from $V$ to $V^*$. It is also conjugate homogeneous by property 3 of semi-inner products. However, it is non-additive unless $V$ reduces to a Hilbert space. As a consequence, a compatible semi-inner product is in general conjugate homogeneous but non-additive with respect to its second variable. Namely,
$$
[f,g+h]_V\ne[f,g]_V+[f,h]_V
$$
in general.

We are ready to present the definition of vector-valued RKBS. Let $\Lambda$ be a Banach space which we shall sometimes call the {\it output space} and $X$ be a prescribed set which is usually called the {\it input space}. A space $\cB$ is called a {\it Banach space of $\Lambda$-valued functions} on $X$ if it consists of certain functions from $X$ to $\Lambda$ and the norm on $\cB$ is compatible with point evaluations in the sense that
$$
\|f\|_\cB=0\mbox{ if and only if }f(x)=0\mbox{ for all }x\in X.
$$
For instance, $L^p([0,1])$, $p\ge1$ is not a Banach space of functions while $C([0,1])$ is. We restrict our consideration to Banach spaces of functions so that point evaluations (usually referred to as ``sampling" in applications) are well-defined.

\begin{defn}
We call $\cB$ a $\Lambda$-valued RKBS on $X$ if both $\cB$ and $\Lambda$ are uniform and $\cB$ is a Banach space of functions from $X$ to $\Lambda$ such that for every $x\in X$, the point evaluation $\delta_x:\cB\to\Lambda$ defined by
$$
\delta_x(f):=f(x),\ \ f\in\cB
$$
is continuous from $\cB$ to $\Lambda$.
\end{defn}

We shall derive a reproducing kernel for so defined a vector-valued RKBS. Throughout the rest of the paper, we let $[\cdot,\cdot]_\cB$ and $[\cdot,\cdot]_\Lambda$ be the unique semi-inner product and $\cJ_\cB$ and $\cJ_\Lambda$ the associated duality mapping on $\cB$ and $\Lambda$, respectively. For two Banach spaces $V_1,V_2$, we denote by $\cM(V_1,V_2)$ the set of all the bounded operators from $V_1$ to $V_2$ and $\cL(V_1,V_2)$ the subset of $\cM(V_1,V_2)$ of those bounded operators that are also linear. When $V_1=V_2$, $\cM(V_1,V_2)$ is abbreviated as $\cM(V_1)$. For each $T\in\cM(V_1,V_2)$, we denote by $\|T\|_{\cM(V_1,V_2)}$ the greatest lower bound of all the nonnegative constants $\alpha$ such that
$$
\|T u\|_{V_2}\le \alpha\|u\|_{V_1}\mbox{ for all }u\in V_1.
$$
When $T$ is also linear, this quantity equals the operator norm $\|T\|_{\cL(V_1,V_2)}$ of $T$ in $\cL(V_1,V_2)$. In those languages, we require that the point evaluation $\delta_x$ on a $\Lambda$-valued RKBS on $X$ belong to $\cL(\cB,\Lambda)$ for all $x\in X$.

\begin{theorem}\label{reproducingkernelthm}
Let $\cB$ be a $\Lambda$-valued RKBS on $X$. Then there exists a unique function $K$ from $X\times X$ to $\cM(\Lambda)$ such that
\begin{description}
\item[(1)] $K(x,\cdot)\xi\in \cB$ for all $x\in X$ and $\xi\in\Lambda$,

\item[(2)] for all $f\in\cB$, $x\in X$, and $\xi\in \Lambda$
\begin{equation}\label{reproducing}
[f(x),\xi]_\Lambda=[f,K(x,\cdot)\xi]_\cB,
\end{equation}
\item[(3)] for all $x,y\in X$
\begin{equation}\label{isbounded}
\|K(x,y)\|_{\cM(\Lambda)}\le\|\delta_x\|_{\cL(\cB,\Lambda)}\|\delta_y\|_{\cL(\cB,\Lambda)}.
\end{equation}

\end{description}
\end{theorem}
\begin{proof}
Let $x\in X$ and $\xi\in\Lambda$. As $\delta_x\in \cL(\cB,\Lambda)$, we see that
\begin{equation}\label{bounedlinearoncb}
\left|[f(x),\xi]_\Lambda\right|\le \|f(x)\|_\Lambda\|\xi\|_\Lambda\le \|\delta_x\|_{\cL(\cB,\Lambda)}\|f\|_\cB\|\xi\|_\Lambda.
\end{equation}
The above inequality together with the linearity of the semi-inner product with respect to its first variable implies that
$$
f\to [f(x),\xi]_\Lambda
$$
is a bounded linear functional on $\cB$. By Lemma \ref{riesz}, there exists a unique function $g_{x,\xi}\in\cB$ such that
\begin{equation}\label{reproducingkernelthmeq1}
[f(x),\xi]_\Lambda=[f,g_{x,\xi}]_\cB.
\end{equation}
Define a function $K$ from $X\times X$ to the set of operators from $\Lambda$ to $\Lambda$ by setting
$$
K(x,y)\xi:=g_{x,\xi}(y),\ \ x,y\in X,\ \xi\in\Lambda.
$$
Clearly, $K$ satisfies the two requirements (1) and (2). It is also unique by the uniqueness of the function $g_{x,\xi}$ satisfying (\ref{reproducingkernelthmeq1}). It remains to show that it is bounded. To this end, we get by (\ref{bounedlinearoncb}) that
$$
\|K(x,\cdot)\xi\|_\cB=\sup_{f\in\cB,\|f\|_\cB\le 1}\left|[f,K(x,\cdot)]_\cB\right|=\sup_{f\in\cB,\|f\|_\cB\le 1}\left|[f(x),\xi]_\Lambda\right|\le \|\delta_x\|_{\cL(\cB,\Lambda)}\|\xi\|_\Lambda.
$$
It follows that
$$
\|K(x,y)\xi\|_\cB\le \|\delta_y\|_{\cL(\cB,\Lambda)}\|K(x,\cdot)\xi\|_\cB\le \|\delta_x\|_{\cL(\cB,\Lambda)}\|\delta_y\|_{\cL(\cB,\Lambda)}\|\xi\|_\Lambda,
$$
which proves (\ref{isbounded}).
\end{proof}

We call the above function $K$ the {\it reproducing kernel} of $\cB$. It coincides with the usual reproducing kernel when $\cB$ is a Hilbert space and $\Lambda=\bC$, and with the vector-valued reproducing kernel when both $\cB$ and $\Lambda$ are Hilbert spaces. We explore basic properties of vector-valued RKBS and its reproducing kernels for further investigation and applications.

Let $(\delta_x)^*$ be the adjoint operator of $\delta_x$ for all $x\in X$. Denote for a Banach space $V$ by $(\cdot,\cdot)_V$ the bilinear form on $V\times V^*$ defined by
$$
(v,\mu)_V:=\mu(v),\ \ v\in V,\ \mu\in V^*.
$$
Thus, $(\delta_x)^*$ is define by
\begin{equation}\label{dualofdeltax}
(f,(\delta_x)^* \xi^*)_\cB=(\delta(x)(f),\xi^*)_\Lambda=(f(x),\xi^*)_\Lambda=[f(x),\xi]_\Lambda,\ \ f\in\cB,\ \xi\in\Lambda.
\end{equation}

\begin{prop}\label{basicproperties}
Let $\cB$ be a $\Lambda$-valued RKBS on $X$ and $K$ its reproducing kernel. Then there holds for all $x,y\in X$ and $\xi,\eta,\tau\in \Lambda$ that
\begin{equation}\label{property1}
[K(x,x)\xi,\xi]_\Lambda\ge0,\ |[K(x,y)\xi,\eta]_\Lambda|\le [K(x,x)\xi,\xi]^{1/2}_\Lambda[K(y,y)\eta,\eta]^{1/2}_\Lambda,
\end{equation}
\begin{equation}\label{property2}
\|K(x,y)\|_{\cM(\Lambda)}\le \|K(x,x)\|_{\cM(\Lambda)}^{1/2}\|K(y,y)\|_{\cM(\Lambda)}^{1/2},
\end{equation}
\begin{equation}\label{property3}
K(x,\cdot)\xi=\cJ_\cB^{-1}\, (\delta_x)^*\cJ_\Lambda (\xi),
\end{equation}
\begin{equation}\label{homogeneouskxy}
K(x,y)(\alpha\xi)=\alpha K(x,y)\xi\mbox{ for all }\alpha\in\bC,
\end{equation}
\begin{equation}\label{property4}
\|K(x,\cdot)\xi\|_\cB\le \|\delta_x\|_{\cL(\cB,\Lambda)}\|\xi\|_\Lambda,\ \ \|K(x,\cdot)\xi\|_\cB\le \|K(x,x)\|_{\cM(\Lambda)}^{1/2}\|\xi\|_\Lambda,
\end{equation}
\begin{equation}\label{property5}
(K(x,\cdot)\xi)^*+(K(x,\cdot)\eta)^*=(K(x,\cdot)\tau)^* \mbox{ whenever }\tau^*=\xi^*+\eta^*,
\end{equation}
\begin{equation}\label{property6}
\span\{(K(x,\cdot)\xi)^*:x\in X,\ \xi\in\Lambda\}\mbox{ is dense in }\cB^*.
\end{equation}
\end{prop}
\begin{proof}
By (\ref{reproducing}),
\begin{equation}\label{normofkxxi}
[K(x,x)\xi,\xi]_\Lambda=[K(x,\cdot)\xi,K(x,\cdot)\xi]_\cB=\|K(x,\cdot)\xi\|_\cB^2\ge0,
\end{equation}
which proves the first inequality in equation (\ref{property1}). For the second one, we use the Cauchy-Schwartz inequality of semi-inner products to get that
$$
\begin{array}{ll}
|[K(x,y)\xi,\eta]_\Lambda|&=|[K(x,\cdot)\xi,K(y,\cdot)\eta]_\cB|\le [K(x,\cdot)\xi,K(x,\cdot)\xi]^{1/2}_\cB[K(y,\cdot)\eta,K(y,\cdot)\eta]^{1/2}_\cB\\
&=[K(x,x)\xi,\xi]^{1/2}_\Lambda[K(y,y)\eta,\eta]^{1/2}_\Lambda.
\end{array}
$$

It follows from (\ref{property1}) that
$$
|[K(x,y)\xi,\eta]_\Lambda|\le \|K(x,x)\xi\|_\Lambda^{1/2}\|\xi\|_\Lambda^{1/2}\|K(y,y)\eta\|_\Lambda^{1/2}\|\eta\|_\Lambda^{1/2}\le \|K(x,x)\|_{\cM(\Lambda)}^{1/2}\|K(y,y)\|_{\cM(\Lambda)}^{1/2}\|\xi\|_\Lambda\|\eta\|_\Lambda.
$$
Since $\|K(x,y)\xi\|_\Lambda=\sup\{|[K(x,y)\xi,\eta]_\Lambda|:\eta\in\Lambda,\|\eta\|_\Lambda=1\}$, we have by the above equation that
$$
\|K(x,y)\xi\|_\Lambda\le \|K(x,x)\|_{\cM(\Lambda)}^{1/2}\|K(y,y)\|_{\cM(\Lambda)}^{1/2}\|\xi\|_\Lambda,
$$
which proves (\ref{property2}).

Turning to (\ref{property3}), we notice for each $f\in\cB$ that
$$
[f,\cJ_\cB^{-1}\, (\delta_x)^*\cJ_\Lambda (\xi)]_\cB=(f,(\delta_x)^*\cJ_\Lambda (\xi))_\cB=(\delta_x(f),\xi^*)_\Lambda=(f(x),\xi^*)_\Lambda=[f(x),\xi]_\Lambda,
$$
which together with (\ref{reproducing}) confirms (\ref{property3}). Since the duality mappings are conjugate homogeneous, we have by (\ref{property3}) that
$$
K(x,\cdot)(\alpha \xi)=\cJ_\cB^{-1}\, (\delta_x)^*\cJ_\Lambda (\alpha\xi)=\alpha\cJ_\cB^{-1}\, (\delta_x)^*\cJ_\Lambda (\xi)=\alpha K(x,\cdot)\xi,
$$
which implies (\ref{homogeneouskxy}).

Recall that the duality mappings $\cJ_\cB$ and $\cJ_\Lambda$ are isometric. Note also that a bounded linear operator and its adjoint have equal operator norms. Using these two facts, we obtain from equation (\ref{property3}) that
$$
\|K(x,\cdot)\xi\|_\cB\le \|(\delta_x)^*\|_{\cL(\Lambda^*,\cB^*)}\|\xi\|_{\Lambda}=\|\delta_x\|_{\cL(\cB,\Lambda)}\|\xi\|_\Lambda,
$$
which is the first inequality in (\ref{property4}). The second one follows immediately from (\ref{normofkxxi}).

Let $\xi,\eta,\tau\in\Lambda$ be such that $\tau^*=\xi^*+\eta^*$. By (\ref{property3}),
$$
(K(x,\cdot)\xi)^*+(K(x,\cdot)\eta)^*=(\delta_x)^*\xi^*+(\delta_x)^*\eta^*=(\delta_x)^*(\xi^*+\eta^*)=(\delta_x)^*\tau^*=(K(x,\cdot)\tau)^*.
$$
Equation (\ref{property5}) hence holds true.

For the last property, let us assume that there exists some $f\in\cB$ that vanishes on $\span\{(K(x,\cdot)\xi)^*:x\in X,\ \xi\in\Lambda\}$. Then
$$
[f(x),\xi]_\Lambda=[f,K(x,\cdot)\xi]_\cB=(f,(K(x,\cdot)\xi)^*)_\cB=0\mbox{ for all }x\in X,\ \xi\in\Lambda,
$$
which implies that $f(x)=0$ for all $x\in X$. As $\cB$ is a Banach space of functions, $f=0$ as a vector in the Banach space $\cB$. Therefore, (\ref{property6}) is true. The proof is complete.
\end{proof}

We observe by the above proposition that the reproducing kernel of a vector-valued RKBS enjoys many properties similar to those of the reproducing kernel of a vector-valued RKHS. However, there are many significant differences due to the nature of a semi-inner product. Firstly, although for all $x,y\in X$, $K(x,y)$ remains a homogeneous bounded operator on $\Lambda$, it is generally non-additive. This can be seen from (\ref{property3}), where $\cJ_\Lambda$ or $\cJ_\cB^{-1}$ is non-additive. Secondly, it is well-known that when $\Lambda$ is a Hilbert space, a function $K:X\times X\to \cL(\Lambda)$ is the reproducing kernel of some $\Lambda$-valued RKHS on $X$ if and only if for all finite $\xi_j\in\Lambda$ and pairwise distinct $x_j\in X$, $j=1,2,\ldots,m$,
\begin{equation}\label{positiverkhs}
\sum_{j=1}^m\sum_{k=1}^m [K(x_j,x_k)\xi_j,\xi_k]_\Lambda\ge0.
\end{equation}
Although (\ref{positiverkhs}) still holds for the reproducing kernel of a vector-valued RKBS when $m\le 2$ and the number field is $\bR$, it may cease to be true once the number of sampling points $m$ exceeds $2$. An example will be constructed in the next section. Finally, the denseness property (\ref{property6}) in the dual space $\cB^*$ does not necessarily imply that
\begin{equation}\label{densinincb}
\overline{\span}\{K(x,\cdot)\xi:x\in X,\ \xi\in \Lambda\}=\cB.
\end{equation}
A negative example will also be given in the next section after we present a construction of vector-valued RKBS through feature maps. Before that, we present another important property of a vector-valued RKBS.

\begin{prop}\label{pointconvergence}
Let $\cB$ be a $\Lambda$-valued RKBS on $X$. Suppose that $f_n\in\cB$, $n\in\bN$ converges to some $f_0\in\cB$ then $f_n(x)$ converges to $f_0(x)$ in the topology of $\Lambda$ for each $x\in X$. The convergence is uniform on the set where $\|K(x,x)\|_{\cM(\Lambda)}$ is bounded.
\end{prop}
\begin{proof}
Suppose that $\|f_n-f\|_\cB$ converges $0$ as $n$ tends to infinity. We get by (\ref{property4}) that
$$
\begin{array}{ll}
\|f_n(x)-f(x)\|_\Lambda&\displaystyle{=\sup_{\xi\in\Lambda,\|\xi\|_\Lambda=1}|[f_n(x)-f(x),\xi]_\Lambda|}\\
&\displaystyle{=\sup_{\xi\in\Lambda,\|\xi\|_\Lambda=1}|[f_n-f,K(x,\cdot)\xi]_\cB|\le \sup_{\xi\in\Lambda,\|\xi\|_\Lambda=1}\|f_n-f\|_\cB \|K(x,\cdot)\xi\|_{\cB}}\\
&\le\|f_n-f\|_\cB \|K(x,x)\|_{\cM(\Lambda)}^{1/2}.
\end{array}
$$
Therefore, $f_n(x)$ converges pointwise to $f(x)$ on $X$ and the convergence is uniform on the set where $\|K(x,x)\|_{\cM(\Lambda)}$ is bounded.
\end{proof}

\section{Feature Map Representations}
\setcounter{equation}{0}

Feature map representations form the most important way of expressing reproducing kernels. To introduce feature maps for the reproducing kernel of a vector-valued RKBS, we need the notion of the generalized adjoint \cite{Koehler} of a bounded linear operator between Banach spaces. Let $V_1,V_2$ be two uniform Banach spaces with the compatible semi-inner products $[\cdot,\cdot]_{V_1}$ and $[\cdot,\cdot]_{V_2}$, respectively. The {\it generalized adjoint} $T^\dag$ of a $T\in\cL(V_1,V_2)$ is an operator in $\cM(V_2,V_1)$ defined by
$$
[Tu,v]_{V_2}=[u,T^\dag v]_{V_1},\ \ u\in V_1,\ v\in V_2.
$$
It can be identified that
$$
T^\dag=\cJ_{V_1}^{-1} T^* \cJ_{V_2}.
$$
Thus, $T^\dag$ is indeed bounded as
$$
\|T^\dag\|_{\cM(V_2,V_1)}=\|T^*\|_{\cL(V_2^*,V_1^*)}=\|T\|_{\cL(V_1,V_2)}.
$$
We are in a position to present a characterization of the reproducing kernel of a vector-valued RKBS.

\begin{theorem}\label{featuremap}
A function $K:X\times X\to\cM(\Lambda)$ is the reproducing kernel of some $\Lambda$-valued RKBS on $X$ if and only if there exists a uniform Banach space $\cW$ and a mapping $\Phi:X\to \cL(\cW,\Lambda)$ such that
\begin{equation}\label{featuremapindeed}
K(x,y)=\Phi(y)\Phi^\dag(x),\ \ x,y\in X,
\end{equation}
and
\begin{equation}\label{densenessfeature}
\overline{\span}\{(\Phi^\dag(x)\xi)^*:x\in X,\ \xi\in\Lambda\}=\cW^*.
\end{equation}
Here $\Phi^\dag$ is the function from $X$ to $\cM(\Lambda,\cW)$ defined by $\Phi^\dag(x):=(\Phi(x))^\dag$, $x\in X$.
\end{theorem}
\begin{proof}
Suppose that $K$ is the reproducing kernel of some $\Lambda$-valued RKBS $\cB$ on $X$. Set $\cW:=\cB$ and define $\Phi:X\to \cL(\cW,\Lambda)$ by
$$
(\Phi(x))(f):=f(x),\ \ f\in \cB,\ \ x\in X.
$$
To identify $\Phi^\dag$, we observe by the reproducing property (\ref{reproducing}) for all $\xi\in\Lambda$ and $f\in\cB$ that
$$
[f,\Phi^\dag(x)\xi]_\cB=[(\Phi(x))f,\xi]_\Lambda=[f(x),\xi]_\Lambda=[f,K(x,\cdot)\xi]_\cB,\ \ x\in X,\ \ \xi\in\Lambda,
$$
which implies that $\Phi^\dag(x)\xi=K(x,\cdot)\xi$ for all $x\in X$ and $\xi\in\Lambda$. Requirement (\ref{densenessfeature}) is fulfilled by (\ref{property6}). By the forms of $\Phi$ and $\Phi^\dag$, we obtain that
$$
\Phi(y)\Phi^\dag(x)\xi=\Phi(y)(K(x,\cdot)\xi)=K(x,y)\xi,
$$
which proves (\ref{featuremapindeed}).

On the other hand, suppose that $K$ is of the form (\ref{featuremapindeed}) in terms of some mapping $\Phi$ satisfying the denseness condition (\ref{densenessfeature}). We shall construct the RKBS that takes $K$ as its reproducing kernel. For this purpose, we let $\cB$ be composed of functions from $X$ to $\Lambda$ of the following form
$$
f_u(x):=\Phi(x)u,\ x\in X\mbox{ for some }u\in\cW.
$$
Since each $\Phi(x)$ is a linear operator, $\cB$ is a linear vector space. We impose a norm on $\cB$ by setting
$$
\|f_u\|_\cB:=\|u\|_\cW,\ \ u\in\cW.
$$
To verify that this is a well-defined norm, it suffices to show that the representer $u$ of a function $f_u\in\cB$ is unique. Assume that $f_u=0$. Then for all $x\in X$ and $\xi\in\Lambda$,
$$
(u,(\Phi^\dag(x)\xi)^*)_\cW=[u,\Phi^\dag(x)\xi]_\cW=[\Phi(x)u,\xi]_\Lambda=[0,\xi]_\Lambda=0,
$$
which combined with (\ref{densenessfeature}) implies that $u=0$. The arguments also show that $\cB$ is a Banach space of functions. Moreover, it is a uniform Banach space as it is isometrically isomorphic to $\cW$. Clearly, we have for each $x\in X$ and $u\in\cW$ that
$$
\|f_u(x)\|_\Lambda=\|\Phi(x)u\|_\Lambda\le \|\Phi(x)\|_{\cL(\cW,\Lambda)}\|u\|_\cW=\|\Phi(x)\|_{\cL(\cW,\Lambda)}\|f_u\|_\cB,
$$
which shows that point evaluations are bounded on $\cB$. We conclude that $\cB$ is a $\Lambda$-valued RKBS on $X$. It remains to prove that $K$ is the reproducing kernel of $\cB$. To this end, we identify the unique compatible semi-inner product on $\cB$ as
$$
[f_u,f_v]_\cB:=[u,v]_\cW,\ \ u,v\in\cW,
$$
and observe for all $u\in\cW$ and $x\in X$ that
$$
[f_u,K(x,\cdot)\xi]_\cB=[f_u,\Phi(\cdot)\Phi^\dag(x)\xi]_\cB=[u,\Phi^\dag(x)\xi]_\cW=[\Phi(x)u,\xi]_\Lambda=[f_u(x),\xi]_\Lambda,
$$
which is what we want. The proof is complete.
\end{proof}

We call the Banach space $\cW$ and the mapping $\Phi$ in Theorem \ref{featuremap} a pair of {\it feature space} and {\it feature map} for $K$, respectively. The proof of Theorem \ref{featuremap} contains a construction of vector-valued RKBS by feature maps, which we pull out separately as a corollary below.

\begin{coro}\label{constructionbyfeaturemap}
Let $\cW$ be a uniform Banach space and $\Phi:X\to\cL(\cW,\Lambda)$ be a feature map of $K$ that satisfies (\ref{featuremapindeed}) and (\ref{densenessfeature}). Then the linear vector space
$$
\cB:=\{\Phi(\cdot)u:\ u\in\cW\}
$$
endowed with the norm
$$
\|\Phi(\cdot)u\|_\cB:=\|u\|_\cW,\ \ u\in\cW
$$
and compatible semi-inner product
$$
[\Phi(\cdot)u,\Phi(\cdot)v]_\cB:=[u,v]_\cW,\ \ u,v\in\cW
$$
is a $\Lambda$-valued RKBS on $X$ with the reproducing kernel $K$ given by (\ref{featuremapindeed}).
\end{coro}

As an interesting application of Corollary \ref{constructionbyfeaturemap}, we shall show that a vector-valued RKBS is always isometrically isomorphic to a scalar-valued RKBS on a different input space.

\begin{coro}\label{isometrictoscalar}
If $\cB$ is a $\Lambda$-valued RKBS on $X$ then the following linear vector space $\tilde{\cB}$ of complex-valued functions $\tilde{f}$ on $\tilde{X}:=X\times\Lambda$ of the form
$$
\tilde{f}(x,\xi):=[f(x),\xi]_\Lambda,\ \ x\in X,\ \xi\in\Lambda,\ f\in\cB
$$
is an RKBS on $\tilde{X}$ with the norm
$$
\|\tilde{f}\|_{\tilde{\cB}}:=\|f\|_\cB,\ \ f\in\cB
$$
and the compatible semi-inner product
$$
[\tilde{f},\tilde{g}]_{\tilde{\cB}}:=[f,g]_\cB,\ \ f,g\in\cB.
$$
The reproducing kernel $\tilde{K}$ of $\tilde{\cB}$ is
$$
\tilde{K}((x,\xi),(y,\eta)):=[K(x,y)\xi,\eta]_\Lambda,\ \ x,y\in X,\ \xi,\eta\in\Lambda.
$$
\end{coro}
\begin{proof}
It suffices to point out that $\tilde{\cB}$ is constructed by Corollary \ref{constructionbyfeaturemap} via the choices
$$
\Lambda:=\bC,\ \cW:=\cB,\ \Phi(x,\xi):=(K(x,\cdot)\xi)^*,\ \ (x,\xi)\in\tilde{X}.
$$
The feature map satisfies the denseness condition by (\ref{property6}).
\end{proof}

We shall next construct by Corollary \ref{constructionbyfeaturemap} simple vector-valued RKBS to show that the reproducing kernel of a general vector-valued RKBS might not satisfy (\ref{positiverkhs}) or (\ref{densinincb}). Let $p,q,r,s\in(1,+\infty)$ satisfy that
\begin{equation}\label{conjugatenumber}
\frac1p+\frac1q=\frac1r+\frac1s=1.
\end{equation}
Here, for the sake of convenience in enumerating elements from a finite set, we set $\bN_l:=\{1,2,\ldots,l\}$ for $l\in\bN$. For each $\gamma\in (1,+\infty)$ and $l\in\bN$, $\ell^l_\gamma$ denotes the Banach space of all vectors $u=(u_j:j\in\bN_l)\in\bC^l$ with the norm
$$
\|u\|_{\ell^l_\gamma}:=\biggl(\sum_{j=1}^l|u_j|^\gamma\biggr)^{1/\gamma}<+\infty.
$$
The space $\ell^l_\gamma$ is a uniform Banach space with the compatible semi-inner product
$$
[u,v]_{\ell^l_\gamma}:=\sum_{j=1}^l\frac{u_j\overline{v_j}|v_j|^{\gamma-2}}{\|v\|_{\ell^l_\gamma}^{\gamma-2}},\ \ u,v\in \ell^l_\gamma.
$$
The dual element $u^*$ of $u\in\ell^l_\gamma$ is hence given by
\begin{equation}\label{dualinbcl}
u^*:=\left(\frac{\overline{v_j}|v_j|^{\gamma-2}}{\|v\|_{\ell^l_\gamma}^{\gamma-2}}:j\in\bN_l\right),\ \ u\in\ell^l_\gamma.
\end{equation}

\noindent {\bf Non-completeness of the linear span of the reproducing kernel in $\cB$.}  We give a counterexample of (\ref{densinincb}) first. Let $m,n\in\bN$. We choose the output space $\Lambda$ and feature space $\cW$ as $\ell^n_p$ and $\ell^m_r$, respectively. Thus, we have that $\Lambda^*=\ell^n_q$ and $\cW^*=\ell^m_s$. The input space will be chosen as a set of $m$ discrete points $X:=\{x_j:j\in\bN_m\}$. A feature map $\Phi:X\to\cL(\cW,\Lambda)$ should satisfy the denseness condition (\ref{densenessfeature}). We note by the definition of the generalized adjoint that this condition is equivalent to
\begin{equation}\label{densenessfeature2}
\overline{\span}\{\Phi^*(x)\xi^*:x\in X,\ \xi\in\Lambda\}=\cW^*,
\end{equation}
where $\Phi^*(x):=(\Phi(x))^*$ for all $x\in X$.

Let us take a close look at equation (\ref{densinincb}). By Corollary \ref{constructionbyfeaturemap}, a general function in $\cB$ is of the form $f_u:=\Phi(\cdot)u$ for some $u\in\cW$. Equation (\ref{densinincb}) does not hold true if and only if there exists a nontrivial $u\in\cW$ such that
$$
[K(x,\cdot)\xi,f_u]_\cB=[\Phi(\cdot)\Phi^\dag(x)\xi,\Phi(\cdot)u]_\cB=[\Phi^\dag(x)\xi,u]_\cW=0,
$$
which in turn is equivalent to that $\span\{\Phi^\dag(x)\xi:x\in X,\ \xi\in\Lambda\}$ is not dense in $\cW$. We conclude that to construct a $\Lambda$-valued RKBS for which (\ref{densinincb}) is not true, it suffices to find a feature map $\Phi:X\to\cL(\cW,\Lambda)$ that satisfies (\ref{densenessfeature2}) but
\begin{equation}\label{notdenseincw}
\overline{\span}\{\Phi^\dag(x)\xi:x\in X,\ \xi\in\Lambda\}\subsetneqq\cW.
\end{equation}

To this end, we find a sequence of vectors $w_j\in\bC^m$ and set
\begin{equation}\label{defphi*}
\Phi^*(x_j)\xi^*:=(\xi^*)_1 w_j,\ \ j\in\bN_m,
\end{equation}
where $(\xi^*)_1$ is the first component of the vector $\xi^*\in\bC^n$. Since for each $j\in\bN_m$, $\Phi^*(x_j)$ is a linear operator from $\Lambda^*$ to $\cW^*$ and both the spaces are finite-dimensional, $\Phi^*(x_j)$ is bounded. We reformulate (\ref{densenessfeature2}) and (\ref{notdenseincw}) to get that they are respectively equivalent to
\begin{equation}\label{discretedense}
\span\{w_j:j\in\bN_m\}=\bC^m
\end{equation}
and
\begin{equation}\label{discretedense}
\span\{\cJ_\cW^{-1}w_j:j\in\bN_m\}\subsetneqq\bC^m.
\end{equation}
Here for a vector $u=(u_j:j\in\bN_m)\in\bC^m$, we get by (\ref{dualinbcl}) that
$$
\cJ_\cW^{-1}u=\left(\frac{\overline{u_j}|u_j|^{s-2}}{\|u\|_{\ell^m_s}^{s-2}}:j\in\bN_m\right).
$$
Therefore, the task reduces to the searching of an $m\times m$ nonsingular matrix $A$ that becomes singular when we apply the function $t\to \overline{t}|t|^{s-2}$ to each of its components. We find two such matrices as shown below
$$
m=4,\ s=4,\ A_1:=\left[
\begin{array}{cccc}
0&8&2&4\\
5&0&5&1\\
5&4&6&9\\
0&9&4&8
\end{array}
\right],\mbox{ and }m=4,\ s=5,\ A_2:=\left[
\begin{array}{cccc}
9&9&9&9\\
8&6&0&2\\
6&9&2&1\\
7&4&9&9
\end{array}
\right].
$$

\noindent {\bf Non-positive-definiteness of the reproducing kernel of $\cB$.} We shall give an example to show that (\ref{positiverkhs}) might not hold true for the reproducing kernel of a vector-valued RKBS when the number $m$ of sampling points exceeds $2$. In fact, we let $m=3$ and $\cB$ be constructed as in the above example with $\{w_j:j\in\bN_3\}$ to be appropriately chosen in the definition (\ref{defphi*}) of $\Phi^*$. Our purpose is to find $w_j\in\bC^3$ and $\xi_j\in\Lambda$, $j\in\bN_3$ such that
\begin{equation}\label{negativedefinite}
\sum_{j=1}^3\sum_{k=1}^3[K(x_j,x_k)\xi_j,\xi_k]_\cB<0.
\end{equation}
We first note for all $j,k\in\bN_3$ that
$$
\begin{array}{ll}
[K(x_j,x_k)\xi_j,\xi_k]_\Lambda&=[\Phi(x_k)\Phi^\dag(x_j)\xi_j,\xi_k]_\Lambda=[\Phi^\dag(x_j)\xi_j,\Phi^\dag(x_k)\xi_k]_\Lambda\\
&=
[(\Phi^\dag(x_k)\xi_k)^*,(\Phi^\dag(x_j)\xi_j)^*]_{\Lambda^*}
=[\Phi^*(x_k)(\xi_k)^*,\Phi^*(x_j)(\xi_j)^*]_{\Lambda^*}.
\end{array}
$$
We shall choose $\xi_j\in\Lambda$ so that $((\xi_j)^*)_1=1$ for each $j\in\bN_3$. With the choice, we obtain by (\ref{defphi*}) and the above equation that
$$
\sum_{j=1}^3\sum_{k=1}^3[K(x_j,x_k)\xi_j,\xi_k]_\cB=
\sum_{j=1}^3\sum_{k=1}^3[w_k,w_j]_{\ell^3_s}.
$$
The conclusion is that for (\ref{negativedefinite}) to hold, it suffices to find $w_j\in\bC^3$, $j\in\bN_3$ that form a basis for $\bC^3$ but
$$
\sum_{j=1}^3\sum_{k=1}^3[w_k,w_j]_{\ell^3_s}<0.
$$
Two examples are shown below
$$
s=4,\ [w_1,w_2,w_3]=\left[\begin{array}{rrr}
4&    -2&    -3\\
3   & -5&     4\\
1  &  -1&     1
\end{array}\right],\mbox{ and }
s=5,\ [w_1,w_2,w_3]=\left[\begin{array}{rrr}
3 &    2 &   -3\\
     2   & -3    & 3\\
    -5    & 0 &    4
\end{array}\right].
$$

\section{Examples of Vector-valued RKBS}
\setcounter{equation}{0}

We present several examples of vector-valued RKBS in this section. The first one of them is applicable to learning a sensing matrix.

\subsection{The space of sensing matrices} Spaces involved in this example are all over the field $\bR$ of real numbers. The input space and output space are chosen by $X:=\bR^d$ and $\Lambda:=\bR^n$. The vector-valued RKBS $\cB$ consists of all the $n\times d$ real matrices. Each $A\in\cB$ is considered to be a function from $\bR^d$ to $\bR^n$ with the point evaluation
$$
A(x):=Ax,\ \ x\in\bR^d.
$$
To find a norm that makes $\cB$ a uniform Banach space, we first point out that a finite-dimensional Banach space $V$ is uniform if and only if its norm is strictly convex. For a proof of this simple fact, see, for example, \cite{ZZacha}. Recall that $\|\cdot\|_V$ is said to be {\it strictly convex} if for all $u,v\in V\setminus\{0\}$, $\|u+v\|_V=\|u\|_V+\|v\|_V$ always implies that $u=\alpha v$ for some $\alpha>0$. Strictly convex norms on $\cB$ include
\begin{itemize}
\item column-wise norms:
\begin{equation}\label{matrixnorm1}
\|A\|_\cB:=G(\|a_1\|_1,\|a_2\|_2,\cdots,\|a_d\|_d),\ \ A\in\cB,
\end{equation}
where for each $j\in\bN_d$, $a_j$ is the $j$-th column of $A$ and $\|\cdot\|_j$ is a strictly convex norm on $\bR^n$, and $G$ is a strictly convex function from $\bR_+^d$ to $\bR_+:=[0,\infty)$ that is strictly increasing with respect to each of its variables and is homogeneous in the sense that
$$
G(\alpha x)=\alpha G(x)\mbox{ for all }x\in\bR_+^d\mbox{ and }\alpha\in\bR_+.
$$
It is straightforward to verify that under the above conditions, (\ref{matrixnorm1}) is indeed a strictly convex norm on $\cB$. An explicit instance is
\begin{equation}\label{matrixnorm2}
\|A\|_\cB:=\|(\|a_j\|_{\ell^n_p}:j\in\bN_d)\|_{\ell^d_r},\ \ A\in\cB,
\end{equation}
where $p,r\in(1,+\infty)$. One can easily transform a column-wise norm $\|\cdot\|_\cB$ into a row-wise norm by equipping $A\in\cB$ with $\|A^T\|_\cB$, where $A^T$ is the transpose of $A$.

\item the $p$-th Schatten norm (see, Section 3.5 of \cite{HornJohnson}):
$$
\|A\|_\cB:=\left(\sum_{j=1}^{\min(n,d)}(\sigma_j(A))^p\right)^{1/p},\ \ A\in\cB,\ \ p\in(1,+\infty),
$$
where $\sigma_j(A)$ is the $j$-th singular value of $A$. The $p$-th Schatten norm belongs to the class of matrix norms that are invariant under multiplication by unitary matrices.
\end{itemize}

We shall look at the reproducing kernel of $\cB$ when it is endowed with the norm (\ref{matrixnorm2}) and the output space $\bR^n$ is equipped with the norm of $\ell^n_\gamma$ for some $\gamma\in(1,+\infty)$. Let $q,s$ be the conjugate number of $p$ and $r$, respectively. In other words, they satisfy (\ref{conjugatenumber}). We proceed by (\ref{reproducing}) that
$$
(Ax,\xi^*)_{\ell^n_\gamma}=[A,K(x,\cdot)\xi]_\cB=(A,(K(x,\cdot)\xi)^*)_\cB,\ \ A\in\cB,\ x\in\bR^d,\ \xi\in\bR^n,
$$
which implies that
\begin{equation}\label{dualofkxxi}
(K(x,\cdot)\xi)^*=\xi^*x^T,\ \ x\in \bR^d,\ \xi\in\bR^n.
\end{equation}
The dual element of $A\in\cB$ is given by
$$
A^*=\frac1{\|A\|_\cB^{r-2}}\left[a_j^*\|a_j\|_{\ell^n_p}^{r-2}:j\in\bN_d\right],
$$
where $a_j^*$ is the dual vector of $a_j$ in $\ell^n_p$. The reproducing kernel of $\cB$ can be derived from the above two equation. Its explicit form is too complicated to be presented. We shall see from the study of regularized learning schemes in vector-valued RKBS that the identification (\ref{dualofkxxi}) of its dual is usually more important.

\subsection{Tensor products of scalar-valued RKBS}

Let $n\in\bN$ and $\cB_j$, $j\in\bN_n$ be scalar-valued RKBS on an input space $X$. We let $\cB$ be the tensor product of $\cB_j$, $j\in\bN_n$. Thus, it consists of $\bC^n$-valued functions of the form $f=(f_j\in\cB_j:j\in\bN_n)$. To define a norm on $\cB$, we choose functions $\cN,\cN^*$ from $\bR_+^n$ to $\bR_+$ that are strictly convex, strictly increasing with respect to each of the variables, homogeneous, and satisfy that $x\to\cN^*(|x|)$ is the dual norm of $x\to\cN(|x|)$ on $\bR^n$. Here, $|x|:=(|x_j|:j\in\bN_n)$ for each $x\in\bR^n$. An example is
$$
\cN(x):=\|x\|_{\ell^n_p},\ \ \cN^*(x):=\|x\|_{\ell^n_q},\ \ x\in\bR_+^n,
$$
where $p,q$ are a pair of conjugate numbers in $(1,+\infty)$. With such two gauge functions, we impose the following norm on $\cB$
\begin{equation}\label{normtensorproduct}
\|f\|_\cB:=\cN(\|f_1\|_{\cB_1},\|f_2\|_{\cB_2},\cdots,\|f_n\|_{\cB_n}),\ \ f\in\cB.
\end{equation}

\begin{prop}\label{tensorproduct}
The tensor product space $\cB$ with the norm (\ref{normtensorproduct}) is a uniform Banach space.
\end{prop}
\begin{proof} We first show that (\ref{normtensorproduct}) defines a uniform convex norm on $\cB$. It is straightforward to verify that it is a norm. Let $\varepsilon$ be a fixed positive number and $f,g\in\cB$ be such that $\|f\|_\cB=\|g\|_\cB=1$ and $\|f-g\|_\cB\ge \varepsilon$. We have that
$$
\begin{array}{ll}
\cN(\|f_1+g_1\|_{\cB_1},\cdots,\|f_n+g_n\|_{\cB_n})&\le \cN(\|f_1\|_{\cB_1}+\|g_1\|_{\cB_1},\cdots,\|f_n\|_{\cB_n}+\|g_n\|_{\cB_n})\\
&\le
\cN(\|f_1\|_{\cB_1},\cdots,\|f_n\|_{\cB_n})+\cN(\|g_1\|_{\cB_1},\cdots,\|g_n\|_{\cB_n}).
\end{array}
$$
As all the norms on $\bR^n$ are equivalent, $\cN$ is continuous on $\bR_+^n$, and vectors $x\in\bR_+^n$ satisfying $\cN(|x|)=1$ form a compact subset in $\bR^n$. We also recall that $\cN$ is strictly increasing with respect to each of its variables and $|x|\to\cN(|x|)$ is a strictly convex norm on $\bR^n$. We conclude from these two facts and the above equation that $\cB$ is uniform convex if there exists some positive constant $\varepsilon'$ independent of $f,g$ such that
$$
\max\{\|f_j\|_{\cB_j}+\|g_j\|_{\cB_j}-\|f_j+g_j\|_{\cB_j}:j\in\bN_n\}\ge \varepsilon'
$$
or
$$
\max\{\left|\|f_j\|_{\cB_j}-\|g_j\|_{\cB_j}\right|:j\in\bN_n\}\ge \varepsilon'.
$$
Assume to the contrary that such a positive constant does not exist. It implies that for all $\beta>0$, there exists $f,g\in\cB$ that satisfy $\|f-g\|_\cB\ge \varepsilon$ and
$$
\|f_j\|_{\cB_j}+\|g_j\|_{\cB_j}-\|f_j+g_j\|_{\cB_j}<\beta,\ \ |\|f_j\|_{\cB_j}-\|g_j\|_{\cB_j}|<\beta\mbox{ for all }j\in\bN_n.
$$
Again, as any two norms on $\bR^n$ are equivalent, the inequality $\|f-g\|_\cB\ge \varepsilon$ implies that $\|f_k-g_k\|_{\cB_k}\ge \varepsilon_0>0$ for some $k\in\bN_n$ and some positive constant $\varepsilon_0$ independent of $f,g$. The conclusion is that there exists some $k\in\bN_n$ and some positive constants $M,\varepsilon_0>0$ such that for all $\beta>0$, there exists $u,v\in\cB_k$ such that $\|u\|_{\cB_k}\le M,\ \|v\|_{\cB_k}\le M$ and
\begin{equation}\label{nodifference}
\|u-v\|_{\cB_k}\ge \varepsilon_0,\quad |\|u\|_{\cB_k}-\|v\|_{\cB_k}|<\beta,\quad \|u\|_{\cB_k}+\|v\|_{\cB_k}-\|u+v\|_{\cB_k}<\beta.
\end{equation}

We shall show that the above equation contradicts the uniform convexity of $\cB_k$. We may choose $\beta$ so small that $\beta<\varepsilon_0/4$. It follows from the first two inequalities of (\ref{nodifference}) that
\begin{equation}\label{tensorproducteq1}
\|u\|_{\cB_k}\ge \frac{\varepsilon_0}4,\quad \|v\|_{\cB_k}\ge \frac{\varepsilon_0}4.
\end{equation}
To proceed, we estimate that
$$
\begin{array}{ll}
\displaystyle{\left\|\frac u{\|u\|_{\cB_k}}-\frac v{\|v\|_{\cB_k}}\right\|_{\cB_k}}&=\displaystyle{\left\|\frac u{\|u\|_{\cB_k}}-\frac v{\|u\|_{\cB_k}}+\frac v{\|u\|_{\cB_k}}-\frac v{\|v\|_{\cB_k}}\right\|_{\cB_k}}\\
&\ge \displaystyle{\frac1{\|u\|_{\cB_k}}\|u-v\|_{\cB_k}-\|v\|_{\cB_k}\left|\frac 1{\|u\|_{\cB_k}}-\frac 1{\|v\|_{\cB_k}}\right|}\\
&\ge \displaystyle{\frac{\varepsilon_0-\beta}{\|u\|_{\cB_k}}}\ge \frac{3\varepsilon_0}{4M}.
\end{array}
$$
By the uniform convexity of $\cB_k$, there exists a positive constant $\delta$ dependent on $\varepsilon_0,M$ and the space $\cB_k$ only such that
\begin{equation}\label{tensorproducteq2}
\left\|\frac u{\|u\|_{\cB_k}}+\frac v{\|v\|_{\cB_k}}\right\|_{\cB_k}<2-\delta.
\end{equation}
Finally, we get by (\ref{nodifference}), (\ref{tensorproducteq1}), and (\ref{tensorproducteq2}) that
$$
\begin{array}{ll}
\|u\|_{\cB_k}+\|v\|_{\cB_k}-\|u+v\|_{\cB_k}&=\displaystyle{\|u\|_{\cB_k}+\|v\|_{\cB_k}-\|u\|_{\cB_k}\left\|\frac u{\|u\|_{\cB_k}}+\frac v{\|v\|_{\cB_k}}+\frac v{\|u\|_{\cB_k}}-\frac v{\|v\|_{\cB_k}}\right\|_{\cB_k}}\\
&\ge \displaystyle{\|u\|_{\cB_k}+\|v\|_{\cB_k}-(2-\delta)\|u\|_{\cB_k}-\|u\|_{\cB_k}\|v\|_{\cB_k}\left|\frac1{\|u\|_{\cB_k}}-\frac1{\|v\|_{\cB_k}}\right|}\\
&\ge \displaystyle{\|u\|_{\cB_k}+\|v\|_{\cB_k}+\left|\|u\|_{\cB_k}-\|v\|_{\cB_k}\right|-(2-\delta)\|u\|_{\cB_k}}\\
&\ge \displaystyle{\delta\|u\|_{\cB_k}\ge \frac{\varepsilon_0\delta}4},
\end{array}
$$
which contradicts to the third inequality of (\ref{nodifference}) as $\beta$ can be arbitrarily small.

It is clear that $\cB^*=\{(f_j^*:j\in\bN_n):f\in\cB\}$ with the norm
$$
\|(f_j^*:j\in\bN_n)\|_{\cB^*}=\cN^*(\|f_1^*\|_{\cB_1^*},\cdots,\|f_n^*\|_{\cB_n^*}).
$$
Similar arguments to those above prove that $\cB^*$ is uniformly convex. By the fact (see \cite{Cudia}) that a Banach space is uniformly Fr\'{e}chet differentiable if and only if its dual is uniformly convex, $\cB$ is uniform.
\end{proof}

We next identify the reproducing kernel of $\cB$ with the following norm
$$
\|f\|_\cB:=\biggl(\sum_{j=1}^n\|f_j\|_{\cB_j}^p\biggr)^{1/p},\ \ f\in\cB.
$$
Let the output space $\bC^n$ be equipped with the norm of $\ell^n_r$ and let $K_j$ be the reproducing kernel of $\cB_j$, $j\in\bN_n$. The unique compatible semi-inner product on $\cB$ is given by
$$
[f,g]_\cB:=\frac{1}{\|g\|_{\cB}^{p-2}}\sum_{j=1}^n\, [f_j,g_j]_{\cB_j}\|g_j\|_{\cB_j}^{p-2},\ \ f,g\in\cB.
$$
The duality mapping on $\cB$ is hence of the form
\begin{equation}\label{dualitymappingtensorproduct}
f^*:=\left(\frac{f_j^*\|f_j\|_{\cB_j}^{p-2}}{\|f\|_{\cB}^{p-2}}:j\in\bN_n\right),\ \ f\in\cB.
\end{equation}
To find an expression for $(K(x,\cdot)\xi)^*$ for $x\in X$ and $\xi\in\bC^n$, we deduce that
$$
[f(x),\xi]_{\ell^n_r}=\frac{1}{\|\xi\|_{\ell^n_r}^{r-2}}\sum_{j=1}^n\overline{\xi_j}|\xi_j|^{r-2}f_j(x)=\frac{1}{\|\xi\|_{\ell^n_r}^{r-2}}\sum_{j=1}^n\overline{\xi_j}|\xi_j|^{r-2}[f_j,K_j(x,\cdot)]_{\cB_j}.
$$
It follows that
\begin{equation}\label{kerneltensorproductdual}
(K(x,\cdot)\xi)^*=\left(\frac{\overline{\xi_j}|\xi_j|^{r-2}}{\|\xi\|_{\ell^n_r}^{r-2}}(K_j(x,\cdot))^*:j\in\bN_n\right),\ \ x\in X,\ \xi\in\bC^n.
\end{equation}
By equations (\ref{dualitymappingtensorproduct}) and (\ref{kerneltensorproductdual}),
$$
\|K(x,\cdot)\xi\|_{\cB}=\frac1{\|\xi\|_{\ell^n_r}^{r-2}}\left(\sum_{j=1}^n \left(|\xi_j|^{r-1}\sqrt{K_j(x,x)}\right)^q\right)^{1/q}, \ \ x\in X,\ \xi\in\bC^n
$$
and
$$
K(x,y)\xi=\left(\frac{\xi_j}{|\xi_j|}K_j(x,y)\left(\frac{\|K(x,\cdot)\xi\|_{\cB}^{p-2}|\xi_j|^{r-1}}{\|\xi\|_{\ell^n_r}^{r-2}K_j(x,x)^{\frac{p-2}2}}\right)^{1/(p-1)}:j\in\bN_n\right),\ \ x,y\in X,\ \xi\in\bC^n.
$$

\subsection{Translation invariant vector-valued RKBS}

An $\bC^n$-valued RKBS $\cB$ on $\bR^d$ is said to be {\it translation invariant} if translations are isometric on $\cB$, namely, if for each $f\in\cB$ and $x\in \bR^d$, $f(\cdot+x)\in\cB$ and $\|f(\cdot+x)\|_\cB=\|f\|_\cB$. It was proved in \cite{XZrefinement} that a scalar-valued RKHS is translation invariant if and only if its reproducing kernel is of the form $\psi(x-y)$ for some scalar-valued function $\psi$. For the Banach space case, as a reproducing kernel alone does not determine its RKBS, we do not have such a characterization. Our purpose in this subsection is to construct a class of translation invariant vector-valued RKBS by the Fourier transform.

Denote by $L^1(\bR^d)$ the Banach space of Lebesgue measurable functions $f$ on $\bR^d$ equipped with the norm
$$
\|f\|_{L^1(\bR^d)}:=\int_{\bR^d}|f(x)|dx.
$$
For $\varphi\in L^1(\bR^d)$, its Fourier transform $\hat{\varphi}$ and inverse Fourier transform $\check{\varphi}$ are respectively given by
$$
\hat{\varphi}(t):=\frac1{(\sqrt{2\pi})^d}\int_{\bR^d}\varphi(x)e^{-ix\cdot t}dx,\ \ t\in\bR^d
$$
and
$$
\check{\varphi}(t):=\frac1{(\sqrt{2\pi})^d}\int_{\bR^d}\varphi(x)e^{ix\cdot t}dx,\ \ t\in\bR^d.
$$
Here $x\cdot t$ is the standard inner product on $\bR^d$.

To start the construction, we let $\phi$ be a nonnegative function in $L^1(\bR^d)$ with $\int_{\bR^d}\phi(x)dx=1$ and denote by $L^p(\bR^d,d\phi)$, $p\in(1,+\infty)$, the Banach space of Lebesgue measurable functions $f$ on $\bR^d$ with the norm
$$
\|f\|_{L^p(\bR^d,d\phi)}:=\left(\int_{\bR^d}|f(x)|^p\phi(x)dx\right)^{1/p}<+\infty.
$$
The feature space $\cW$ is chosen as
$$
\cW:=\{u=(u_1,\ldots,u_n):u_j\in L^p(\bR^d,d\phi),\ j\in\bN_n\}
$$
endowed with the norm
$$
\|u\|_{\cW}:=\biggl(\sum_{j=1}^n \|u_j\|_{L^p(\bR^d,d\phi)}^p\biggr)^{1/p}.
$$
Its dual space $\cW^*$ is given by
$$
\cW^*=\{w=(w_1,\ldots,w_n):w_j\in L^q(\bR^d,d\phi),\ j\in\bN_n\}
$$
with the norm
$$
\|w\|_{\cW^*}:=\biggl(\sum_{j=1}^n \|w_j\|_{L^q(\bR^d,d\phi)}^q\biggr)^{1/q}.
$$
The bilinear form on $\cW\times\cW^*$ is
$$
(u,w)_\cW=\sum_{j=1}^n\int_{\bR^d}u_j(x)w_j(x)\phi(x)dx,\ \ u\in\cW,\ w\in\cW^*.
$$
Moreover, the dual element of $u\in\cW$ is
$$
u^*=\left(\frac{u_j^*\|u_j\|_{L^p(\bR^d,d\phi)}^{p-2}}{\|u\|_\cW^{p-2}}:j\in\bN_n\right).
$$
By Proposition \ref{tensorproduct}, $\cW$ is a uniform Banach space. Our feature map $\Phi:\bR^d\to \cL(\cW,\bC^n)$ is then defined by
$$
\Phi(x)u:=S(u\phi)\hat{\,}(x),\ \ x\in\bR^d,\ u\in\cW,
$$
where $S$ is an invertible $n\times n$ matrix and $(u\phi)\hat{\,}:=((u_j\phi)\hat{\,}:j\in\bN_n)$. The map $\Phi$ is well-defined as $f\phi\in L^1(\bR^d)$ for all $f\in L^p(\bR^d,d\phi)$ by the H\"{o}lder inequality. We also notice that $\Phi(x)$ is continuous from $\cW$ to $\bC^n$ for each $x\in\bR^d$ by the fact that
$$
|(f\phi)\hat{\,}(x)|\le \|f\phi\|_{L^1(\bR^d)}\le \|f\|_{L^p(\bR^d,d\phi)}\mbox{ for all } f\in L^p(\bR^d,d\phi).
$$
One sees that the adjoint operator $\Phi^*:\bR^d\to \cL(\bC^n,\cW^*)$ is given by
$$
\Phi^*(x)(\eta)=\frac{e^{-ix\cdot t}}{(\sqrt{2\pi})^d}S^T\eta,\ \ x\in\bR^d,\ \eta\in\bC^n.
$$
Clearly, the denseness condition (\ref{densenessfeature2}) is satisfied. The equivalent condition (\ref{densenessfeature}) hence holds true. We obtain by Corollary \ref{constructionbyfeaturemap} that
$$
\cB:=\{f_u:=S(u\phi)\hat{\,}:u\in\cW\}
$$
with the norm $\|f_u\|_{\cB}:=\|u\|_{\cW}$ and compatible semi-inner product
$$
[S(u\phi)\hat{\,},S(v\phi)\hat{\,}\,]_\cB=[u,v]_\cW=\sum_{j=1}^n\frac1{\|v\|_{\cW}^{p-2}}\int_{\bR^d}u_j(x)\overline{v_j(x)}|v_j(x)|^{p-2}\phi(x)dx
$$
is a $\bC^n$-valued RKBS. It is translation invariant because for all $y\in\bR^d$ and $u\in\cW$
$$
\|S(u\phi)\hat{\,}(\cdot+y)\|_\cB=\|S(e^{-iy\cdot t}u\phi)\hat{\,}\|_\cB=\|e^{-y\cdot t}u\|_\cW=\|u\|_\cW=\|S(u\phi)\hat{\,}\|_\cB.
$$
To understand the reproducing kernel of $\cB$, we present the dual space of $\cB$
$$
\cB^*=\{S(u^*\phi)\check{\,}:u\in\cW\}
$$
with the norm, compatible semi-inner product and bilinear form
$$
\|S(u^*\phi)\check{\,}\|_{\cB^*}=\|u^*\|_{\cW^*},\quad [S(u^*\phi)\check{\,},S(v^*\phi)\check{\,}\,]_{\cB^*}=[v,u]_\cW,\quad (S(u\phi)\hat{\,},S(v^*\phi)\check{\,})_\cB=(u,v^*)_{\cW}.
$$
With these preparations, we identify by (\ref{reproducing}) that
$$
(K(x,\cdot)\xi)^*=S(v_{x,\xi}^*\phi)\check{\,},\ \ x\in\bR^d,\ \xi\in\bC^n,
$$
where
$$
v_{x,\xi}^*(t):=\frac{e^{-ix\cdot t}}{(\sqrt{2\pi})^d}S^T\xi^*,\ \ t\in\bR^d
$$
and $\xi^*$ is the dual element of $\xi$ in $\bC^n$ under a strictly convex norm. By the above two equations,
$$
(K(x,\cdot)\xi)^*(y)=\frac1{(\sqrt{2\pi})^d}SS^T\xi^*\hat{\phi}(x-y),\ \ x,y\in \bR^d,\ \xi\in\bC^n.
$$
We also derive that
$$
K(x,y)\xi=\frac{\|S^T\xi^*\|_{\ell^n_q}^{\frac{p-2}{p-1}}}{(\sqrt{2\pi})^d}S\left(\frac{\overline{(S^T\xi^*)_j}}{|(S^T\xi^*)_j)|^{\frac{p-2}{p-1}}}:j\in\bN_n\right)^T\hat{\phi}(y-x),\ \ x,y\in\bR^d,\ \xi\in\bC^n.
$$
We remark that when $p=2$, $\bC^n$ is endowed with the standard Euclidean norm $\|\cdot\|$, and $\phi$ is the Gaussian function, $K$ becomes the Gaussian kernel for $\bC^n$-valued RKHS
$$
K(x,y)=SS^*\exp\left(-\frac{\|x-y\|^2}2\right),\ \ x,y\in\bR^d,
$$
which confirms the validity of the above construction.

\section{Multi-task Learning with Banach Spaces}
\setcounter{equation}{0}

We discuss the applications of vector-valued RKBS to the learning of vector-valued functions from finite samples. Specifically, suppose that the unknown target function is from the input space $X$ to an output space $\Lambda$ and the observations of the function on given sampling points $\{x_j:j\in\bN_m\}\subseteq X$ are available. The observation at $x_j$, $j\in\bN_m$ could be $f(x_j)$ or the application of some continuous linear functional in $\Lambda^*$ on $f(x_j)$. And it is usually corrupted by noise in practice. To handle the noise and have a good generalization error, we shall follow the regularization methodology. For notational simplicity, let $\bx:=(x_j:j\in\bN_m)\in X^m$ and $f(\bx):=(f(x_j):j\in\bN_m)\in\Lambda^m$. A general learning scheme has the following form
\begin{equation}\label{regularization1}
\inf_{f\in\cB}Q(f(\bx))+\lambda \Psi(\|f\|_\cB),
\end{equation}
where $\cB$ is a chosen $\Lambda$-valued RKBS on $X$, $Q:\Lambda^{m}\to\bR_+$ is a loss function, $\lambda$ is a positive regularization parameter, and $\Psi:\bR_+\to\bR_+$ is called a {\it regularizer}. We are concerned with the existence and uniqueness, representation, and solving of the minimizer of (\ref{regularization1}). Before moving on to these topics, let us see some examples of learning schemes of the form (\ref{regularization1}):
\begin{itemize}
\item[---] Regularization networks
\begin{equation}\label{regularizationnetwork}
Q(f(\bx)):=\sum_{j=1}^m\|f(x_j)-\xi_j\|_\Lambda^2,\ \ \Psi(\|f\|_\cB):=\|f\|_\cB^2,
\end{equation}
where $\xi_j\in\Lambda$, $j\in\bN_m$ are observed outputs of $f$ at $\bx$. In general, one may use
\begin{equation}\label{generllossfunctionnorm}
Q(f(\bx))= P(\|f(x_1)-\xi_1\|_\Lambda,\cdots,\|f(x_m)-\xi_m\|_\Lambda),
\end{equation}
where $P$ is a function from $\bR_+^m\to\bR_+$. A particular choice of $P$ leads to the support vector machine regression.

\item[---] Support vector machine regression
$$
\Lambda:=\bR^n,\ Q(f(\bx))=\sum_{j=1}^m \max(0,\|f(x_j)-\xi_j\|_{\ell^n_1}-\varepsilon),
$$
where $\varepsilon$ is a positive constant standing for the tolerance level.

\item[---] Spectral learning: when $\cB$ is the space of sensing matrices introduced in the last section with a unitarily invariant matrix norm, (\ref{regularization1}) is the special spectral learning considered in \cite{AMP2010}.
\end{itemize}

\subsection{Existence and Uniqueness}

The weak topology is the weakest topology on a Banach space $V$ such that elements in $V^*$ remain continuous on $V$. A sequence $u_n\in V$, $n\in\bN$, is said to converge weakly to $u_0\in V$ if for each $\mu\in V^*$, $\mu(u_n)$ converges to $\mu(u_0)$. We call a regularizer $\Psi:\bR_+\to\bR_+$ {\it admissible} if it is continuous and nondecreasing on $\bR_+$ with
\begin{equation}\label{tendstoinfinity}
\lim_{t\to\infty}\Psi(t)=+\infty.
\end{equation}

\begin{prop}\label{firstsufficient}
If $Q:\Lambda^{m}\to\bR_+$ is continuous with respect to each of its variables under the weak topology on $\Lambda$ and $\Psi$ is an admissible regularizer then (\ref{regularization1}) has at least a minimizer.
\end{prop}
\begin{proof}
Arguments similar to those in the proof of Proposition 4 in \cite{ZZjogo} still apply to the vector-valued case considered here.
\end{proof}

When $\Lambda$ is finite-dimensional, any two topologies on it are equivalent. Thus, continuity under the weak topology is equivalent to continuity with respect to the norm of $\Lambda$.
\begin{coro}
Let $\cB$ be finite-dimensional. If $Q:\Lambda^{m}\to\bR_+$ is continuous with respect to each of its variables and $\Psi$ is an admissible regularizer then (\ref{regularization1}) has at least a minimizer.
\end{coro}

We next deal with the case when the loss function has the form (\ref{generllossfunctionnorm}).
\begin{prop}\label{secondsufficient}
If $P:\bR_+^m\to\bR_+$ is continuous on $\bR_+^m$ and nondecreasing with respect to each of its variables and the regularizer $\Psi$ is admissible then
\begin{equation}\label{regularization2}
\inf_{f\in\cB}P(\|f(x_1)-\xi_1\|_\Lambda,\cdots,\|f(x_m)-\xi_m\|_\Lambda)+\lambda\Psi(\|f\|_\cB)
\end{equation}
has a minimizer.
\end{prop}
\begin{proof}
Set
$$
\cE(f):=P(\|f(x_1)-\xi_1\|_\Lambda,\cdots,\|f(x_m)-\xi_m\|_\Lambda)+\lambda\Psi(\|f\|_\cB),\ \ f\in\cB.
$$
and $\varepsilon_0:=\inf_{f\in\cB}\cE(f)$. Using the arguments similar to those in \cite{ZZjogo}, we can find a sequence $f_n\in\cB$, $n\in\bN$ that is weakly convergent to some $f_0\in\cB$, and some $\alpha>0$ such that $\|f_0\|_\cB\le\alpha$ and $\|f_n\|_\cB\le \alpha$ for all $n\in\bN$. Moreover, for any $\epsilon>0$ there exists some $N\in\bN$ such that for $n>N$,
\begin{equation}\label{secondsufficienteq1}
\Psi(\|f_n\|_\cB)\ge \Psi(\|f_0\|_\cB)-\epsilon.
\end{equation}
Since $f_n$ converges weakly to $f_0$, by (\ref{reproducing})
$$
\lim_{n\to\infty}[f_n(x_j)-\xi_j,f_0(x_j)-\xi_j]_\Lambda=[f_0(x_j)-\xi_j,f_0(x_j)-\xi_j]_\Lambda\mbox{ for all }j\in\bN_m.
$$
It implies by the Cauchy-Schwartz inequality of semi-inner products that for any $\delta>0$ there exists some $N'\in\bN$ such that for $n>N'$
\begin{equation}\label{secondsufficienteq2}
\|f_n(x_j)-\xi_j\|_\cB\ge \|f_0(x_j)-\xi_j\|_\cB-\delta\mbox{ for all }j\in\bN_m.
\end{equation}
Since
$$
\|f_0(x_j)-\xi_j\|_\cB,\ \|f_n(x_j)-\xi_j\|_\cB\le \max\{\alpha\|\delta_{x_j}\|_{\cL(\cB,\Lambda)}+\|\xi_j\|_\Lambda:j\in\bN_m\}
$$
and $\Psi$ is uniformly continuous on compact subsets of $\bR_+^m$ and is nondecreasing with respect to each of its variables, we get by (\ref{secondsufficienteq2}) that
$$
P(\|f_n(x_1)-\xi_1\|_\Lambda,\cdots,\|f_n(x_m)-\xi_m\|_\Lambda)\ge P(\|f_0(x_1)-\xi_1\|_\Lambda,\cdots,\|f_0(x_m)-\xi_m\|_\Lambda)-\epsilon
$$
for sufficiently large $n$. This combined with (\ref{secondsufficienteq1}) proves that $f_0$ is a minimizer of (\ref{regularization2}).
\end{proof}

For uniqueness of the minimizer, we have the following routine result.
\begin{prop}
If $Q$ is convex on $\Lambda^m$ and $\Psi$ is strictly increasing and strictly convex then (\ref{regularization1}) has at most one minimizer.
\end{prop}
\begin{proof}
It is straightforward that the function mapping $f\in\cB$ to $Q(f(\bx))+\lambda \Psi(\|f\|_\cB)$ is strictly convex on $\cB$.
\end{proof}

We close this subsection with the following corollary to the above propositions.
\begin{coro}
Let $\cB$ be a $\Lambda$-valued RKBS on $X$. Then $\inf_{f\in\cB}\cE(f)$ has a unique minimizer for the following choices of regularization functionals:
$$
\cE(f)=\sum_{j=1}^m \|f(x_j)-\xi_j\|_\Lambda^p+\lambda \|f\|_\cB^r,\quad p\in[1,+\infty),\ r\in(1,+\infty),
$$
$$
\cE(f)=\sum_{j=1}^m \max(0,\|f(x_j)-\xi_j\|_\Lambda-\varepsilon)+\lambda \|f\|_\cB^r,\quad r\in(1,+\infty),\ \varepsilon>0.
$$
\end{coro}

\subsection{The representer theorem}

We study the representation of the minimizer of (\ref{regularization1}) by the reproducing kernel $K$ of $\cB$. The result, known as the representer theorem in the scalar-valued and vector-valued RKHS cases, was due to \cite{KW} and \cite{MP2005}, respectively. For more references on this subject for the RKHS case, see \cite{AMP2009,ScAS} and the references cited therein. We established the representer theorem for scalar-valued RKBS in \cite{ZXZ,ZZjogo}. The representer theorem is closely related to the minimal norm interpolation. We start with examining the latter problem.

Let $\bx:=(x_j:j\in\bN_m)\in X^m$ be a fixed set of sampling points. Denote for each $\bz:=(\eta_j:j\in\bN_m)\in\Lambda^m$ by $\cI_\bz$ the set of functions $f\in\cB$ that satisfy the interpolation condition $f(\bx)=\bz$. We need two notations for the proof of the representer theorem for the minimal norm interpolation. For a subset $A$ of Banach space $V$, $A^\perp$ stands for the set of all the continuous linear functionals on $V$ that vanish on $A$, and for $B\subseteq V^*$, ${\,}^\perp B:=\{u\in V:\mu(u)=0\mbox{ for all }\mu\in B\}$.



\begin{lemma}\label{MNIlemma}
Let $\bz\in\Lambda^m$. If $\cI_\bz$ is nonempty then the minimal norm interpolation problem
\begin{equation}\label{MNI}
\inf\{\|f\|_\cB:f\in\cI_\bz\}
\end{equation}
has a unique minimizer. A function $f_0\in\cB$ is the minimizer of (\ref{MNI}) if and only if $f(\bx)=\bz$ and
\begin{equation}\label{representerforMNI}
f_0^*\in \overline{\span}\left\{(K(x_j,\cdot)\xi)^*:j\in\bN_m,\ \xi\in\Lambda\right\}.
\end{equation}
\end{lemma}
\begin{proof}
Clearly, $\cI_\bz$ is a closed convex subset of $\cB$. A minimizer of (\ref{MNI}) is the best approximation in $\cI_\bz$ to the origin $0$ of $\cB$. It is well-known that a closed convex subset in a uniform convex Banach space has a unique best approximation to a point in the same space. By this fact, (\ref{MNI}) has a unique minimizer. It is also trivial that $f_0\in\cI_\bz$ is the minimizer if and only if
$$
\|f_0+g\|_\cB\ge \|f_0\|_\cB\mbox{ for all }g\in\cI_0.
$$
By the characterization of best approximation by the semi-inner product established in \cite{Giles}, the above equation holds if and only if
$$
[g,f_0]=0\mbox{ for all }g\in\cI_0,
$$
which can be equivalently expressed as $f_0^*\in (\cI_0)^\perp$. Note that $g\in\cI_0$ if and only if
$$
[g,K(x_j,\cdot)\xi]_\cB=[g(x_j),\xi]_\Lambda=0\mbox{ for all }j\in\bN_m\mbox{ and }\xi\in\Lambda,
$$
which is equivalent to that
$$
g\in {\,}^\perp\left\{( K(x_j,\cdot)\xi)^*:j\in\bN_m,\ \xi\in\Lambda\right\}.
$$
We conclude that $f_0\in\cI_\bz$ is the minimizer of (\ref{MNI}) if and only if
$$
f_0^*\in \left({\,}^\perp\left\{( K(x_j,\cdot)\xi)^*:j\in\bN_m,\ \xi\in\Lambda\right\}\right)^\perp.
$$
By the Hahn-Banach theorem, for each $B\in\cB^*$, $({\,}^\perp B)^\perp=\overline{\span} B$. The proof is hence complete.
\end{proof}

The above lemma enables us to prove the main result of the section without much effort.

\begin{theorem}\label{representermain}
Suppose that (\ref{regularization1}) has at least a minimizer. If the regularizer is nondecreasing then (\ref{regularization1}) has a minimizer that satisfies (\ref{representerforMNI}). If $\Psi$ is strictly increasing then every minimizer of (\ref{regularization1}) must satisfy (\ref{representerforMNI}).
\end{theorem}
\begin{proof}
Let $f\in\cB$ be a minimizer of (\ref{regularization1}). We let $f_0$ be the minimizer of
\begin{equation}\label{representermaineq1}
\min\{\|g\|_\cB:g\in\cI_{f(\bx)}\}.
\end{equation}
Then $\|f_0\|_\cB\le \|f\|_\cB$ and $f_0(\bx)=f(\bx)$. It follows that $Q(f_0(\bx))=Q(f(\bx))$ while $\Psi(\|f_0\|_\cB)\le \Psi(\|f\|_\cB)$ as $\Psi$ is nondecreasing. Therefore, $f_0$ is a minimizer of (\ref{regularization1}). By Lemma \ref{MNIlemma}, $f_0$ satisfies (\ref{representerforMNI}).

Suppose that $\Psi$ is strictly increasing and $f\in\cB$ does not satisfy (\ref{representerforMNI}). Again, we let $f_0\in\cB$ be the minimizer of (\ref{representermaineq1}). As $f$ does not satisfy (\ref{representerforMNI}), $f\ne f_0$ by Lemma \ref{MNIlemma}. Thus, $\|f\|_\cB>\|f_0\|_\cB$. The consequence is that while $Q(f(\bx))=Q(f_0(\bx))$, $\Psi(\|f\|_\cB)> \Psi(\|f_0\|_\cB)$ because $\Psi$ is strictly increasing. Therefore, $f$ can not be the minimizer of (\ref{regularization1}). The proof is complete.
\end{proof}

\subsection{Characterization equations}

We consider the solving of the regularized learning scheme (\ref{regularization1}) in this subsection. We try to make use of the representer theorem. To this end, we note that the output space $\Lambda$ is usually finite-dimensional in practice. Let us assume that (\ref{regularization1}) has a unique minimizer $f_0$, $\dim(\Lambda)=n<+\infty$, and $\{e_l^*:l\in\bN_n\}$ is a basis for $\cB^*$. In this case, we see by property (\ref{property5}) of the reproducing kernel $K$ that $f_0$ has the form
\begin{equation}\label{repreenterfinitedimension}
f_0^*=\sum_{j=1}^m (K(x_j,\cdot)\eta_j)^*
\end{equation}
for some $\eta_j\in\Lambda$, $j\in\bN_m$. It hence suffices to find the finite model parameters $\eta_j$'s in order to obtain $f_0$. To this end, one may substitute (\ref{repreenterfinitedimension}) into (\ref{regularization1}) to convert the original minimization problem in a potentially infinite-dimensional Banach space into one about the finitely many parameters $\eta_j$'s. We next show how the reformulation can be done under the finite-dimensionality assumption on $\Lambda$. As each $\xi\in\Lambda$ is uniquely determined by $\{[\xi,e_l]_\Lambda:l\in\bN_n\}$. We may rewrite the regularization functional as
\begin{equation}\label{regularization3}
\min_{f\in\cB}\cR(([f(\xi_j),e_l]_\Lambda:j\in\bN_m,\ l\in\bN_n))+\lambda\Psi(\|f\|_\cB)
\end{equation}
for some function $\cR:\bC^{m\times n}\to\bR_+$. By (\ref{reproducing}) and (\ref{sipondual})
$$
[f(\xi_j),e_l]_\Lambda=[f,K(x_j,\cdot)e_l]_\cB=[(K(x_j,\cdot)e_l)^*,f^*]_{\cB^*}.
$$
For the regularizer part, we have by (\ref{equalnormduality}) that $\|f\|_\cB=\|f^*\|_{\cB^*}$. Therefore, the parameters $\eta_j$'s in (\ref{repreenterfinitedimension}) are the minimizer of
$$
\min_{\tau\in\Lambda^m}\cR\left(\left(\left[(K(x_j,\cdot)e_l)^*,\sum_{k=1}^m(K(x_k,\cdot)\tau_k)^*\right]_{\cB^*}:j\in\bN_m,\ l\in\bN_n\right)\right)+\lambda\Psi\biggl(\biggl\|\sum_{j=1}^m (K(x_j,\cdot)\tau_j)^*\biggr\|_{\cB^*}\biggr).
$$
Unlike the RKHS case, the above minimization problem is usually non-convex with respect to $\tau_j^*$ or $\tau_j$ even when $\cR$ and $\Psi$ are both convex. The reason is that a semi-inner product is generally non-additive with respect to its second variable.

In some occasions, one is able to derive a characterization equation for the minimization problem (\ref{regularization1}), which together with the representer theorem constitutes a powerful tool in converting the minimization into a system of equations about the model parameters in the representer theorem. We shall derive characterization equations for the particular example of (\ref{regularization1})
\begin{equation}\label{regularization4}
\min_{f\in\cB}\sum_{j=1}^m \varphi(\|f(x_j)-\xi_j\|_\Lambda)+\lambda\Psi(\|f\|_\cB),
\end{equation}
where $\xi_j$ stands for the observation of the target function at $x_j$ for $j\in\bN_m$, and $\varphi$ is a chosen loss function from $\bR_+$ to $\bR_+$. We shall assume that both $\varphi$ and $\Psi$ are continuously differentiable and
\begin{equation}\label{derivativevanishingat0}
\lim_{t\to 0^+}\frac{\varphi'(t)}{t}=0.
\end{equation}
For convenience, we make the convention that $0/0:=0$. The next two results hold for any $\Lambda$ regardless of its dimension.
\begin{theorem}\label{characterizationeq}
Let $\Psi$ and $\varphi$ be continuously differentiable on $\bR_+$ with (\ref{derivativevanishingat0}). A function $f_0\ne0$ is the minimizer of (\ref{regularization4}) if and only if
\begin{equation}\label{charaeqnonzero}
\lambda \frac{\Psi'(\|f_0\|_\cB)}{\|f_0\|_\cB}f_0^*+\sum_{j=1}^m\frac{\varphi'(\|f_0(x_j)-\xi_j\|_\cB)}{\|f_0(x_j)-\xi_j\|_\cB}(K(x_j,\cdot)(f_0(x_j)-\xi_j))^*=0.
\end{equation}
The zero function is the minimizer of (\ref{regularization4}) if and only if
\begin{equation}\label{charaeqzero}
\|T\|_{\cB^*}\le \lambda \Psi'(0),
\end{equation}
where
$$
T:=\sum_{j=1}^m\frac{\varphi'(\|\xi_j\|_\Lambda)}{\|\xi_j\|_\Lambda}(K(x_j,\cdot)\xi_j)^*.
$$
\end{theorem}
\begin{proof}
The proof is similar to that for the scalar-valued RKBS case in \cite{ZZjogo}. One only needs to handle the semi-inner product in vector-valued RKBS carefully.
\end{proof}

In the sequel, we discuss the application of the above theorem to the regularization networks
\begin{equation}\label{regularizationnetwork2}
\min_{f\in\cB}\sum_{j=1}^m\|f(x_j)-\xi_j\|_\Lambda^2+\lambda \|f\|_\cB^2.
\end{equation}
To this end, we say that the point evaluations on $\cB$ at $x_j$, $j\in\bN_m$ are {\it essentially linearly independent} if for all $\eta_j\in\Lambda$, $j\in\bN_m$
$$
\sum_{j=1}^m[f(x_j),\eta_j]_\Lambda=0\mbox{ for all }f\in\cB
$$
necessitates that $\eta_j=0$ for each $j\in\bN_m$. By (\ref{reproducing}), $\delta_{x_j}$, $j\in\bN_m$ are essentially linearly independent if and only if
$$
\sum_{j=1}^m (K(x_j,\cdot)\eta_j)^*=0
$$
implies that $\eta_j=0$ for each $j\in\bN_m$.

\begin{coro}\label{regularizationnetworkcoro}
Suppose that the point evaluations on $\cB$ at $x_j$, $j\in\bN_m$ are essentially linearly independent. Then $f_0$ is the minimizer of the regularization network (\ref{regularizationnetwork2}) if and only if it is of the form (\ref{repreenterfinitedimension}) where the parameters $\eta_j$'s satisfy
\begin{equation}\label{simpleequation}
\lambda \eta_j+f_0(x_j)-\xi_j=0\mbox{ for all }j\in\bN_m.
\end{equation}
\end{coro}
\begin{proof}
For the regularization network (\ref{regularizationnetwork2}), (\ref{charaeqnonzero}) and (\ref{charaeqzero}) are equivalent to each other when $f_0=0$. By Theorem \ref{characterizationeq}, $f_0$ is the minimizer of (\ref{regularizationnetwork2}) if and only if
\begin{equation}\label{regularizationnetworkcoroeq1}
\lambda f_0^*+\sum_{j=1}^m \left(K(x_j,\cdot)(f_0(x_j)-\xi_j)\right)^*=0.
\end{equation}
Thus, $f_0$ has the form (\ref{repreenterfinitedimension}). Since $\delta_{x_j}$, $j\in\bN_m$ are essentially linearly independent, (\ref{regularizationnetworkcoroeq1}) is equivalent to that the parameters $\eta_j$'s in (\ref{repreenterfinitedimension}) satisfy (\ref{simpleequation}). The proof is complete.
\end{proof}

Similarly, one may substitute the representer theorem into the characterization equations (\ref{charaeqnonzero}) and (\ref{simpleequation}) to reduce the minimization problem to the solving of a system of equations about the parameters $\eta_j$'s. Again, due to the non-additivity of a semi-inner product with respect to its second variable, the resulting equations are generally nonlinear about the parameters. We conduct the reformulation when $\Lambda$ is of finite dimension $n\in\bN$ and $\{e_l^*:l\in\bN_n\}$ forms a basis for $\Lambda^*$. In this case, (\ref{simpleequation}) can be reformulated as
$$
\lambda [\eta_j,e_l]_\Lambda+\biggl[(K(x_j,\cdot)e_l)^*,\sum_{k=1}^m (K(x_k,\cdot)\eta_k)^*\biggr]_{\cB^*}=[\xi_j,e_l],\ \ j\in\bN_m,\ l\in\bN_n.
$$
We shall leave the solving of the resulting non-convex minimization problem and nonlinear equations about the parameters in the representer theorem for future study.


\end{document}